\documentclass[12pt,journal]{IEEEtran}
\onecolumn
\usepackage{mathrsfs}
\usepackage{algorithmic}
\usepackage{algorithm}
\usepackage{amsmath}
\usepackage{amsthm}
\usepackage{graphicx}
\usepackage{amssymb}
\usepackage{epstopdf}
\usepackage{enumerate}
\usepackage{longtable,tabularx,float}
\usepackage{cite}
\usepackage{mathcomp}
\usepackage{multirow}
\usepackage{supertabular}
\usepackage{stmaryrd}
\usepackage{color}
\usepackage{url}
\usepackage{makecell}
\usepackage[OT2,OT1]{fontenc}
\usepackage{bm}
\usepackage{bbm}
\usepackage{caption}
\newtheorem{lemma}{Lemma}[section]

\newtheorem{theorem}{Theorem}[section]

\newtheorem{problem}{Problem}[section]

\newtheorem{example}{Example}[section]
\newtheorem{construction}{Construction}[section]
\newtheorem{claim}{Claim}[section]

\newtheorem{remark}{Remark}[section]
\newtheorem{conjecture}{Conjecture}[section]

\begin{document}

\title{On supersaturation for oddtown and eventown}

\author{Xin~Wei, Yuhao~Zhao, Xiande~Zhang~and~Gennian~Ge
\thanks{This project was supported by the National Key Research and Development Program of China under Grant 2020YFA0712100,  Grant 2018YFA0704703 and Grant 2020YFA0713100, the National Natural Science Foundation of China under Grant 11971325, Grant 12171452 and Grant 12231014,  and Beijing Scholars Program.}
 \thanks{X. Wei ({\tt weixinma@mail.ustc.edu.cn}) and Y. Zhao ({\tt zhaoyh21@mail.ustc.edu.cn}) are with the School of Mathematical Sciences, University of Science and Technology of China, Hefei, 230026, Anhui, China.}
 \thanks{X. Zhang ({\tt drzhangx@ustc.edu.cn}) is with the School of Mathematical Sciences, University of Science and Technology of China, Hefei, 230026, Anhui, China, and  Hefei National Laboratory, Hefei, 230088, China.}
 \thanks{G. Ge ({\tt gnge@zju.edu.cn}) is with the School of Mathematical Sciences, Capital Normal University, Beijing, 100048, China.
}
%
%
%

}
\maketitle

\begin{abstract}
  We study the supersaturation problems of oddtown and eventown. Given a family $\mathcal A$ of subsets of an $n$ element set, let $op(\mathcal A)$ denote the number of distinct pairs $A,B\in \mathcal A$ for which $|A \cap B|$ is odd. We show that if $\mathcal A$ consists of $n+s$ odd-sized subsets, then $op(\mathcal A)\geq s+2$, which is tight when $s\le n-4$. This disproves a conjecture by O'Neill on the supersaturation problem of oddtown.  For the supersaturation problem of eventown, we show that for large enough $n$, if $\mathcal A$ consists of $2^{\lfloor n\slash 2\rfloor}+s$ even-sized subsets, then $op(\mathcal A)\ge s\cdot2^{\lfloor n\slash 2\rfloor-1} $ for any positive integer $s\le 2^{\lfloor\frac n 8\rfloor}\slash n$. This partially proves a conjecture by O'Neill on the supersaturation problem of eventown. Previously, the correctness of this conjecture was only verified for $s=1$ and $2$. We further provide a twice weaker lower bound in this conjecture for eventown, that is  $op(\mathcal{A})\ge s\cdot 2^{\lfloor n/2\rfloor-2}$ for general $n$ and $s$ by using discrete Fourier analysis. Finally, some asymptotic results for the lower bounds of $op(\mathcal A)$ are given when $s$ is large for both problems.

%
%

\end{abstract}


\section{Introduction}

In extremal set theory, given a finite family (i.e., a collection of subsets) $\mathcal F$ and a restriction on the intersection of two subsets, the {\it restricted intersection problem} asks for the maximum size of a subfamily $\mathcal A\subset\mathcal F$ such that any two different members of $\mathcal A$ satisfy the restricted  intersection. Many fundamental and classical results in extremal combinatorics can be stated as a restricted intersection problem. Let $[n]:=\{1, 2, \ldots, n\}$, $2^{[n]}$ denote the collection of all subsets of $[n]$, and $\binom{[n]}{k}$ denote the collection of all $k$-subsets of $[n]$. Then the celebrated Erd\H{o}s-Ko-Rado theorem on intersecting families~\cite{erdos1961intersection} can be viewed as a solution to the restricted intersection problem with $\mathcal F=\binom{[n]}{k}$ for $n\ge 2k$ when restricting empty pairwise intersections. As another example, Sperner's theorem on antichains~\cite{sperner1928satz} states that the maximum size of a subfamily $\mathcal A$ of $\mathcal F=2^{[n]}$ with restricted intersection $A\backslash B {=A\cap B^{c}} \neq \emptyset$ for any $A, B\in\mathcal A$ is $\binom n {\lfloor\frac n 2\rfloor}$. There are several other well-studied restricted intersection problems, such as $L$-intersecting families and bounded symmetric differences. For more information, one can refer to~\cite{fisher1940174, frankl2017stability, gerbner2018extremal, snevily1995generalization, kleitman1966combinatorial, gao2022stability}.


In this paper, we focus on the \emph{oddtown} and \emph{eventown} problems, which are also restricted intersection problems. Both of them share the same restriction that  intersections of every two different members have even size. The difference is that oddtown requires the family $\mathcal F$ to consist of all odd-sized subsets of ${[n]}$, while eventown  requires  all even-sized subsets.
Formally, let $\mathcal A=\{A_1, A_2, \ldots, A_m\}$ be a family of subsets of $[n]$. We say $\mathcal A$ is an \emph{oddtown} (resp. \emph{eventown}) \emph{family} if all its sets have odd (resp. even) size and $$|A_i\cap A_j|\text{ is even for }1\le i<j\le m.$$ The maximum size of an oddtown family is $n$, and the maximum size of an eventown family is $2^{\lfloor\frac n 2\rfloor}$, which were determined by  Berlekamp~\cite{berlekamp1969subsets} and Graver~\cite{graver1975boolean} independently. Their methods highlighted the linear algebra method~\cite{babai1988linear} in extremal combinatorics. Numerous extensions and variants of the oddtown and eventown problems
can be found in the literature \cite{deza1983functions, o2022note, vu1999extremal, szabo2005exact,FRANKL1983215, vu1997extremal, sudakov2018two}, such as extending modulo $2$ to modulo general $\ell$, which is known as  $\ell$-even\slash oddtown problem~\cite{babai1988linear, FRANKL1983215}, and extending  pairwise restricted intersections to $k$-wise restricted intersections \cite{vu1997extremal, sudakov2018two}.


Recently in \cite{o2021towards}, O'Neill initiated the study of supersaturation problem for oddtown and eventown:
if  $\mathcal A\subset 2^{[n]}$ is a family of more than $n$ odd-sized subsets, or a family of  more than  $2^{\lfloor\frac n 2\rfloor}$ even-sized subsets, how many pairs of members in $\mathcal A$  must
violate the intersecting restriction, that is, have an odd number of elements in common?
 Supersaturation versions  of other foundational
problems in extremal set theory have also  attracted a lot of attention recently.  For example,  works like~\cite{balogh2018kleitman, balogh2018structure, leader2003set, das2015sperner, das2016minimum}   gave the supersaturation versions for  Erd\H{o}s-Ko-Rado theorem and  Sperner's theorem.

%

For a given set family $\mathcal A$, the {\it odd pair number}, denoted by $op(\mathcal A)$, is the number of  pairs of distinct members $A, B\in \mathcal A$ such that $|A\cap B|$ is odd. In \cite{o2021towards}, O'Neill constructed a family $\mathcal A=\{\{i\}\}_{i\in [n]}\cup \mathcal C_s$, where $\mathcal C_s$ consists of exactly $s$ members from the extremal oddtown family:  a collection of vertex disjoint $K_4^{(3)}$
- i.e., all triples on four vertices. It is easy to check that $\mathcal A$ is a family of $n+s$  odd-sized subsets,  and $op(\mathcal A)=3s$. O'Neill proved that this is the best possible result for $s=1$:
\begin{theorem}[\hspace{-0.01em}\cite{o2021towards}]\label{thm_old_oddtown}
Let $n\ge 1$ and $\mathcal A\subset 2^{[n]}$ consists of odd-sized subsets with $|\mathcal A|\ge n+1$. Then $op(\mathcal A)\ge 3$.
\end{theorem}

O'Neill further conjectured that:
\begin{conjecture}[\hspace{-0.01em}\cite{o2021towards}]\label{conj_odd_town}
Let $n\ge 1$ and fix $1\le s\le n$. If $\mathcal A\subset 2^{[n]}$ is a family of odd-sized subsets with $|\mathcal A|\ge n+s$, then $op(\mathcal A)\ge 3s$.
\end{conjecture}

Our first main result is to show that Conjecture~\ref{conj_odd_town} is not true. In fact, when $n\ge s+4$, we can construct an odd-sized family of size $n+s$ but with odd pair number $s+2$, which is much smaller than the lower bound in Conjecture~\ref{conj_odd_town}. We further show that $s+2$ is best possible for any $s$ and $n\geq s+4$. The statement is summarized below. For brevity, we use the term ``odd-sized (resp. even-sized) family'' to present a family consisting of odd-sized (resp. even-sized) subsets.
\begin{theorem}\label{thm_for_oddtown}
Let $n\ge 1$ and fix $1\le s\le n-4$.  Any odd-sized family $\mathcal A\subset 2^{[n]}$ with $|\mathcal A|\ge n+s$ must satisfy $op(\mathcal A)\ge s+2$, and the lower bound is tight.
\end{theorem}

Theorem~\ref{thm_for_oddtown}  focuses on the supersaturation problem of oddtown when $|\mathcal A|$ exceeds the corresponding extremal number $n$ by some $s$ smaller than $n$. One can also ask the same question when $s$ is larger than $n$. Although we do not compute the exact value for the smallest odd pair number, we have the following asymptotic result when $n$ goes to infinity. The standard asymptotic notations like $o$, $O$ and $\Theta$ are used in this paper to compare two functions when $n$  goes to infinity, and all logarithms are under base $2$ by default.

\begin{theorem}\label{theorem_oddtown_approximation}
Given some positive integer valued function $s=s(n)$, denote $M_o(s, n)$ as the minimum number of $op(\mathcal A)$ among all odd-sized subfamily $\mathcal A\subset 2^{[n]}$ with size $|\mathcal A|=n+s$.
If $s=cn+ o(n)$ for some constant $c>0$,
\begin{equation}\label{eq-mo2}
   M_o(s, n)=\big(\binom{\lfloor c\rfloor+1}2+(\lfloor c\rfloor+1)(c- \lfloor c\rfloor)\big)n+ o(n).
 \end{equation}
\end{theorem}

For the eventown case, O'Neill~\cite{o2021towards} constructed a family of $2^{ n\slash2}+s$ even-sized subsets whose odd pair number is $s\cdot 2^{ n\slash2-1}$ when $n$ is doubly even. His construction is as follows.
 Suppose $n=2k=4\ell$ and let $X_1, X_2, \ldots, X_\ell\subset[n]$ be pairwise disjoint subsets with $X_i=\{4i-3, 4i-2, 4i-1, 4i\}$. For each $X_i$, define four subsets $A_{2i-1}=\{4i-3, 4i-2\}$, $A_{2i}=\{4i-1, 4i\}$, $B_{2i-1}=\{4i-3, 4i\}$, and $B_{2i}=\{4i-2, 4i-1\}$. Then define two collections,
\begin{equation}
\mathcal A=\{\cup_{j\in J}A_j: J\subset [k]\}\text{ and }\mathcal B=\{\cup_{j\in J}B_j: J\subset [k]\}.
\end{equation}
Observe that both $\mathcal A$ and $\mathcal B$ are extremal eventown families. Moreover for each $B\in\mathcal B\backslash\mathcal A$, $op(\mathcal A\cup \{B\})=2^{k-1}$. Note that $|\mathcal B\backslash\mathcal A|=2^k-2^\ell$ by linear algebra. For any $s\in [2^k-2^\ell]$, consider $\mathcal A'$ formed by $\mathcal A$ and $s$ distinct members from $\mathcal B\backslash\mathcal A$. Then $|\mathcal A'|=2^k+s$ and $op(\mathcal A')=s\cdot 2^{k-1}$. O'Neill~\cite{o2021towards} proved that this is best possible for $s=1,2$ and further conjectured this is true for a large range of $s$.

\begin{conjecture}\label{conj_even_town}
Let $n\ge 1$ and fix $1\le s\le 2^{\lfloor n\slash2\rfloor}- 2^{\lfloor n\slash4\rfloor}$. If $\mathcal A\subset 2^{[n]}$ consists of even-sized subsets with $|\mathcal A|\ge 2^{\lfloor n\slash2\rfloor}+s$, then $op(\mathcal A)\ge s\cdot 2^{\lfloor n\slash2\rfloor-1}$.
\end{conjecture}

Recently progress on Conjecture~\ref{conj_even_town} was made in \cite{antipov2022lovasz}, where half of the lower bound for even $n$ (but much weaker bound for odd $n$) and general $s$ was proved  by spectral analysis.

\begin{theorem}[\hspace{-0.01em}\cite{antipov2022lovasz}]\label{thm_old_eventown}
Let $n, s$ be positive integers. If an even-sized family $\mathcal A\subset 2^{[n]}$ satisfies $|\mathcal A|\ge 2^{ n\slash2}+s$, then $op(\mathcal A)\ge s\cdot 2^{\lfloor n\slash2\rfloor-2}$.
\end{theorem}

Our next contribution is to show that Conjecture~\ref{conj_even_town} is true for a wide range of $s$ and for sufficiently large $n$ by using extremal graph theory. We state it below.
\begin{theorem}\label{thm_for_eventown}
Let $n$ be a large enough integer and fix $s\in[2^{\lfloor\frac n 8\rfloor}\slash n]$. Any even-sized family $\mathcal A\subset 2^{[n]}$ with $|\mathcal A|\ge 2^{\lfloor n\slash2\rfloor}+s$ satisfies $op(\mathcal A)\ge s\cdot 2^{\lfloor n\slash2\rfloor-1}$.
\end{theorem}

Similar to the supersaturation problem of oddtown, one can consider the case when $s$ is larger than $2^{\lfloor n\slash2\rfloor}$. When $| \mathcal{A}| \ge 2^{(1-\epsilon)n}$ for some $\epsilon\in(0,1/2)$, O'Neill \cite{o2021towards} proposed the following problem.
\begin{problem}
Let $\epsilon\in\left( 0,\frac{1}{2}\right) $ and $n$ be sufficiently large. Determine the maximum value $f_n(\epsilon)$ so that if $\mathcal{A}\subseteq 2^{[n]}$ is an even-sized family with $\left| \mathcal{A}\right| \ge 2^{(1-\epsilon)n}$, then $op(\mathcal{A})\ge f_n(\epsilon)\binom{|\mathcal{A}|}{2}$.
\end{problem}

We show that when $| \mathcal{A}| \ge 2^{(1-\epsilon)n}$ for any given $\epsilon\in(0,1/2)$, the density of $op(\mathcal A)$ always approaches $\frac12$. The formal statement is as follows.

\begin{theorem}\label{thm_eventownlar}
Let $\epsilon\in\left( 0,\frac{1}{2}\right) $ and $n\ge {1\slash\epsilon}$. We have
$$
		\frac{1}{2}\left( 1-2^{\left( \epsilon-\frac{1}{2}\right) n} \right) \le f_n(\epsilon) \le
\begin{cases}
\frac{1}{2}, & n\text{ odd;}\\
\frac12(1-\frac1{2^{n-1}-1}), & n\text{ even.}
\end{cases}
$$
Hence for fixed $\epsilon\in \left( 0,\frac{1}{2} \right)$, we have $\lim\limits_{n\rightarrow \infty} f_n(\epsilon) = \frac{1}{2}$.
	\end{theorem}

By improving an intermediate result in the proof of Theorem~\ref{thm_eventownlar}, we show that half of the lower bound in Conjecture~\ref{conj_even_town} is true for general $n$ and $s$. Thus it completes the result in Theorem~\ref{thm_old_eventown} from \cite{antipov2022lovasz}.

\begin{theorem}\label{thm_complete_old_eventown}
For any positive integers $n$ and $s$, let $\mathcal A\subset 2^{[n]}$ be an even-sized family with $|\mathcal A|\ge 2^{\lfloor n\slash2\rfloor+s},$ then $op(\mathcal A)> s\cdot2^{\lfloor n\slash2\rfloor-2}$.
\end{theorem}

 The proofs of Theorem~\ref{thm_eventownlar} and Theorem~\ref{thm_complete_old_eventown} apply
Fourier analysis on finite abelian groups~\cite{stein2011fourier}.

The rest of this paper is organized as follows. In Section~\ref{sec_preliminary}, we introduce some necessary notations and basic results in Fourier analysis, and then give a glance at the supersaturation problem of Tur\'an theorem. We prove Theorems~\ref{thm_for_oddtown}-\ref{theorem_oddtown_approximation} in Section~\ref{sec_supersatuation_oddtown}, and Theorems~\ref{thm_for_eventown}-\ref{thm_complete_old_eventown}  in Section~\ref{sec_supersaturation_eventown}.  Finally, a conclusion and some remarks are listed in Section~\ref{sec-conc}.
%

\section{Preliminary}\label{sec_preliminary}


We begin with some useful notations used throughout this paper.
For two integers $a\le b$, we always use $[a, b]$ to denote the set of consecutive integers from $a$ to $b$, i.e., $\{a, a+1, \ldots, b\}$. For any set $S$ and positive integer $k$, $\binom S k$ stands for the collection of all $k$-sized subsets of $S$, and $2^S$ stands for the collection of all subsets of $S$. For a subset $A\subseteq[n]$, let $\bm v_A\in \mathbb F_2^n$ be the characteristic vector of $A$, that is, for any $j\in [n]$, $\bm v_A(j)=1$ if and only if $j\in A$. For a vector subset $W$ of $\mathbb F_2^n$, let $W^\perp=\{\bm v\in \mathbb F_2^n: (\bm v, \bm u)=0, ~ \forall\bm u \in W\}$,  where $(\cdot , \cdot)$ means the natural inner product in $\mathbb F_2^n$. If $W=\{\bm u\}$ for some $\bm u$, we simply write $\bm u^\perp$ instead of $\{\bm u\}^\perp$.


Given a  graph $H=(V, E)$, let $v(H)$ and $e(H)$ denote the numbers of vertices and edges in $H$, respectively. For a vertex set $U\subset V$, the induced subgraph of $H$ on set $U$ is denoted by $H[U]$, which has vertex set $U$ and edge set $E\cap\binom U 2$. Similarly, for two  disjoint $U_1, U_2\subset V$, let $H[U_1, U_2]$ denote the induced bipartite subgraph of $H$ with two parts $U_1$ and $U_2$ and let $e(U_1, U_2)$ denote the number of edges in $H[U_1, U_2]$.
The complement of $H$ is denoted as $\overline H$. For any vertex subset $U\subset V$, the neighborhood of $U$ on $H$, denoted by $N_H(U)$, is the set of vertices in $V\setminus U$ which is adjacent to at least one vertex in $U$. If $U=\{v\}$, we simply write $N_H(v)$. The degree of $v$ is denoted by $d_H(v)$. The subscript $H$ will be omitted if $H$ is clear.


For any set family $\mathcal A\subseteq 2^{[n]}$, the \emph{odd pair graph} of $\mathcal A$, denoted by $H(\mathcal A)$, is constructed in the following way: the vertices are members in $\mathcal A$, and  two different members $A, B\in \mathcal A$ are adjacent if and only if $|A\cap B|$ is odd. Consequently,  the number of edges in $H(\mathcal A)$ is exactly $op(\mathcal A)$.
There is another way to construct the odd pair graph: let $V\subseteq \mathbb F_2^n$, and define graph $H(V)$ as vertex set $V$ and edge set $\{\{\bm u, \bm v\}: \bm u\ne \bm v\in V; (\bm u, \bm v)=1\}$. It is easy to check that under bijection: $A\mapsto \bm v_A$ from $2^{[n]}$ to $\mathbb F_2^n$, these two constructions are equivalent when $V$ is the image of $\mathcal A$. Thus when considering the odd pair graph, a vertex, its corresponding subset $A\subseteq [n]$, and its characteristic vector $\bm v_\mathcal A \in \mathbb F_2^n$ are seen as the same object, and sometimes they share the same notation. Denote the odd pair graph of the largest collection $2^{[n]}$ as $H_0=H(2^{[n]})=H(\mathbb F_2^n)$. Under these notations, for any subfamilies $\mathcal A_1\subseteq \mathcal A_2\subseteq 2^{[n]}$, $H(\mathcal A_1)$ is the induced subgraph of $ H(\mathcal A_2)$ on $\mathcal A_1$, and specially, $H(\mathcal A_1)=H_0[\mathcal A_1]$.

\subsection{Fourier analysis}
Given a finite abelian group $G$ with the additive notation, a {\it character} on $G$ is defined as a complex valued function $\chi: G\to S^1$, where $S^1$ means the unit circle in the complex plane, such that for all $a, b\in G$, $\chi(a+b)=\chi(a)\chi(b)$. Denote the set of all characters of $G$ as $\widehat G$. Then $\widehat G=Hom(G, S^1)$ is a multiplicative abelian group with multiplication defined by $(\chi_1\cdot\chi_2)(a)=\chi_1(a)\chi_2(a)$ for all $a\in G$. Since group $G$ is finite abelian, one can use the fundamental theorem for finite abelian groups~\cite{kurzweil2004theory} to show that $\widehat{G}\cong G$, where $\cong$ means the group isomorphism. In particular, $|\widehat G|=|G|$.

Denote the vector space of complex-valued functions on $G$ as $L(G)$. Define a Hermitian inner product on $L(G)$ by
$$(f, g)=\frac 1{|G|}\sum_{a\in G}f(a)\overline{g(a)}, \text{ for any }f, g\in L(G).$$
It can be proved that the characters of $G$ form an orthonormal basis of $L(G)$ with the inner product defined above. Given a function $f\in L(G)$, define the {\it Fourier transform} of $f$ as a function $\hat f: \widehat G\to \mathbb C$, such that for any $\chi\in\widehat G$, $$\hat f(\chi)=(f, \chi)=\frac 1{|G|}\sum_{a\in G}f(a)\overline{\chi(a)},$$
 which leads to the following {\it Fourier inversion formula},
$$f=\sum_{\chi\in\widehat G}(f, \chi)\chi= \sum_{\chi\in\widehat G}\hat f(\chi)\chi.$$
Define the norm of $f\in L(G)$ from the inner product as $\Vert f\Vert=(f, f)^{1\slash 2}$. Then the following is the corresponding {\it Plancherel formula,}
$$\Vert f\Vert^2=\sum_{\chi\in\widehat G}|\hat f(\chi)|^2.$$

\subsection{Supersaturation of Tur\'an theorem}
Next we introduce some results on the supersaturation problem of Tur\'an theorem~\cite{erdHos1983supersaturated}.  Denote $ex(n, K_t)$ as the maximum possible number of edges in a graph on $n$ vertices which does not contain $K_t$ as a subgraph. Tur\'an \cite{turan1941extremalaufgabe} proved that for any positive integers $n$ and $r$, $ex(n, K_{r+1})=e(T_{n, r})$, where $T_{n, r}$  is the unique $n$ vertex $r$-partite complete graph with each part of size $\lfloor\frac n r\rfloor$ or $\lceil\frac n r\rceil$. An approximate form of Tur\'an theorem shows that:
\begin{theorem}\label{coro_turan}
For positive integers $n$ and $r$,
$$ex(n, K_{r+1})\le\big(\frac{r-1}r \big)\frac{n^2}2.$$
\end{theorem}
If the edge density of some graph $G$ on $n$ vertices exceeds the density of Tur\'an  graph $T_{n, r}$, which is $\frac{r-1}r$, 
then a large amount of forbidden structures, i.e., $K_{r+1}$ as subgraphs of $G$, will emerge.
\begin{theorem}[\hspace{-0.01em}\cite{erdHos1983supersaturated}]\label{thm_graph_supersaturation}
For every $\epsilon>0$, there exist some $\delta=\delta(\epsilon)>0$ and integer $n_0=n_0(\epsilon)$ such that every graph on $n\ge n_0$ vertices with at least $(\frac{r-1}r+\epsilon)\binom n 2$ edges contains at least $\delta n^{r+1}$ copies of $K_{r+1}$ as a subgraph.
\end{theorem}

\section{Supersaturation problem of oddtown}\label{sec_supersatuation_oddtown}

\subsection{Proof of Theorem~\ref{thm_for_oddtown}}
We first prove the lower bound.
\begin{lemma}\label{lemma_oddtown_lower_bound}
Let $n\ge s\ge 1$. Any odd-sized family $\mathcal A\subset 2^{[n]}$ with $|\mathcal A|\ge n+s$ must satisfy $op(\mathcal A)\ge s+2$.
\end{lemma}
\begin{proof}
We prove it by induction on the value of $s$. For the base case when $s=1$, the statement is true from Theorem~\ref{thm_old_oddtown}. Assume that the statement is true for all $s\le k$ for some $k>0$. Consider the case when $s=k+1$.

Suppose on the contrary, there exists some odd-sized family $\mathcal A$ on $[n]$ for some $n\ge k+1$ with $|\mathcal A|\ge n+k+1$ and $op(\mathcal A)\le k+2$. Since $op(\mathcal A)\ge 1$, we can choose a member $A\in\mathcal A$ such that $A$ is not an isolated vertex in $H(\mathcal A)$. Denote $\mathcal A'=\mathcal A\backslash\{A\}$, which is also an odd-sized collection on $[n]$, but with $|\mathcal A'|\ge n+k$ and $op(\mathcal A')\le op(\mathcal A)-1\le k+1$. This leads to a contradiction to the statement when $s=k$.
\end{proof}

It is left to give a construction of odd-sized family $\mathcal A$ with $|\mathcal A|=n+s$ and $op(\mathcal A)=s+2$ when $n\geq s+4$. We first give a construction when $n=s+4$.

\begin{construction}\label{cons_A_s}
For any integer $s\ge 1$, let $n=s+4$. Construct a family $\mathcal A_s$ of $2s+4$ odd-sized subsets of $[n]$ as follows.
\begin{itemize}
\item[(1)] There are six special subsets $\{1, 2, 3\}$, $\{1, 2, 4\}$, $\{1, 2, 5\}$, $\{1, 3, 4\}$, $\{1, 3, 5\}$, $\{1, 4, 5\}$. We call the collection of those six subsets the center of $\mathcal A_s$, denoted by $\mathcal{C}$, which is irrelevant to the value of $s$.
\item[(2)] The remaining $2s-2$ subsets form $s-1$ pairs: $\{i+5\}$ and $\{ 2, 3, 4, 5, i+5\}$, $i\in [s-1]$.
\end{itemize}
\end{construction}

\begin{example}
When $s=3$, $n=7$, the family $\mathcal A_3$ consists of ten odd-sized subsets in $[7]$: six subsets in the center, i.e., $\{1, 2, 3\}$, $\{1, 2, 4\}$, $\{1, 2, 5\}$, $\{1, 3, 4\}$, $\{1, 3, 5\}$ and $\{1, 4, 5\}$, and four additional subsets forming two pairs, i.e., $\{6\}$, $\{2, 3, 4, 5, 6\}$ and $\{7\}$, $\{2, 3, 4, 5, 7\}$.
\end{example}


Next, we show that the odd pair number of $\mathcal A_s$ is $s+2$.

\begin{lemma}\label{lemma_A_s_is_good}
For any integer $s\ge 1$, 
$op(\mathcal A_s)=s+2$.
\end{lemma}
\begin{proof}
It is easy to check that $op(\mathcal{C})=3$, and the two subsets in each of the $s-1$ pairs out of the center have odd-sized intersection. We claim that no other two members in $\mathcal A_s$ can have odd-sized intersections, and hence $op(\mathcal A_s)=op(\mathcal{C})+(s-1)=s+2$. There are only two possible cases need to be checked: (1) exactly one member is in the center; (2) the two members are from different pairs.

For the first case, without loss of generality, suppose $A_1\in \mathcal{C}$ and $A_2$ is from the $i$th pair for some $i\in [s-1]$. Then $A_1=\{1, x, y\}$ for some $x\ne y\in [2,  5]$. If $A_2={i+5}$, then $A_1\cap A_2=\emptyset$. If $A_2=\{2, 3, 4, 5, i+5\}$, then $A_1\cap A_2=\{x, y\}$. Both situations lead to even-sized intersections.

For the second case, there exist some $i\ne j\in [s-1]$ such that $A_1$ and $A_2$ are in the $i$th and the $j$th pair, respectively. Then $A_1$ is one of $\{i+5\}$ and $\{2, 3, 4, 5, i+5\}$, while $A_2$ is one of $\{j+5\}$ and $\{2, 3, 4, 5, j+5\}$. All of the four situations lead to even-sized intersections.
%
\end{proof}

From the proof of Lemma~\ref{lemma_A_s_is_good}, we see that the odd pair graph $H(\mathcal A_s)$ is a perfect matching on $2s+4$ vertices, which does not contain an extremal oddtown family of size $s+4$. Now we extend the result in Lemma~\ref{lemma_A_s_is_good} to any $n\geq s+4$.
\begin{lemma}\label{lemma_s+2_is_achivable}
For any fixed integers $s\ge 1$ and $n\ge s+4$, there exists an odd-sized family $\mathcal A\subset 2^{[n]}$ with $|\mathcal A|=n+s$ and $op(\mathcal A)=s+2$.
\end{lemma}
\begin{proof}
Consider $\mathcal A$ consisting of the family $\mathcal A_s$ together with the following subsets: $\{i\}$, $i\in [s+5,n]$. Then  $\mathcal A$ is an odd-sized family on $[n]$ with cardinality $(2s+4)+(n-s-4)=n+s$. It is easy to see those new-added subsets with single elements do not add new odd-sized intersections. So by Lemma~\ref{lemma_A_s_is_good}, $op(\mathcal A)=op(\mathcal A_s)=s+2$.
\end{proof}

Combining Lemma~\ref{lemma_oddtown_lower_bound} and Lemma~\ref{lemma_s+2_is_achivable}, Theorem~\ref{thm_for_oddtown} is proved.

We remark that the construction in the proof of Lemma~\ref{lemma_s+2_is_achivable} is not the only extremal construction under the equivalence of permutation. There is another method to construct good odd-sized families with large size and small odd pair number based on the existence of eventown families.
\begin{construction}~\label{cons_A_m_I}
For given integers $n> m\ge 1$, divide $[n]$ into two parts $[m]$ and $[m+1, n]$. Let $\mathcal E$ be an eventown family on $[m]$. Define  $\mathcal A(\mathcal E, m,n)$ as the product family of $\mathcal E$ and $\{\{i\}: i\in[m+1, n]\}$. That is,
$$\mathcal A(\mathcal E, m,n)=\{E\cup \{i\}: E\in\mathcal E, i\in [m+1, n]\}.$$
\end{construction}
Then $\mathcal A(\mathcal E, m,n)$ is a family on $[n]$ of $|\mathcal E|(n-m)$ odd-sized subsets. For any two different sets $A_1= E_1\cup \{i_1\}$ and $A_2= E_2\cup \{i_2\}$ from $\mathcal A(\mathcal E, m,n)$, where $E_1, E_2\in \mathcal E$ and $i_1, i_2\in[m+1, n]$, the size $|A_1 \cap A_2|$ is odd if and only if $i_1= i_2$. Thus, $op(\mathcal A(\mathcal E, m,n))=\binom {|\mathcal E|} 2 (n-m)$.

For any $n\ge 3$, if we set $\mathcal E=\{\emptyset, \{1, 2\}\}$, $\mathcal A(\mathcal E, 2, n)$ is with size $2n-4$ and $op(\mathcal A(\mathcal E, 2, n))=n-2$, which meets Lemma~\ref{lemma_oddtown_lower_bound} when $s=n-4$. Moreover, for any $1\le s\le n-4$, $\mathcal A(\mathcal E, 2, n)\backslash\{\{1, 2, i\}: i> s+4\}$ contains $n+s$ subsets and its odd pair number  is $s+2$. So it is  another example achieving the lower bound in Lemma~\ref{lemma_oddtown_lower_bound}. This example is not equivalent to the constructions from Lemma~\ref{lemma_s+2_is_achivable}, because of the different size distributions of the two families.  Construction~\ref{cons_A_m_I} will be used to prove Theorem~\ref{theorem_oddtown_approximation}.

\subsection{Asymptotic result of oddtown}
The lower bound of Theorem~\ref{theorem_oddtown_approximation} comes from the following lemma.
\begin{lemma}\label{lemma_odd_approx_lower}
For any odd-sized family $\mathcal A\subset 2^{[n]}$ with $|\mathcal A|=n+s$, we have  \begin{equation}\label{eq-app}
op(\mathcal A)\ge \binom{\lfloor s\slash n\rfloor+1}2 n+(\lfloor s\slash n\rfloor+1)(s-n\lfloor s\slash n\rfloor).
 \end{equation}
\end{lemma}
\begin{proof}
We define a sequence of disjoint subcollections of $\mathcal A$: $\mathcal A_1, \mathcal A_2, \ldots$ recursively. First let $\mathcal A_1$ be a maximum oddtown subfamily of $\mathcal A$, i.e., the largest subfamily of $\mathcal A$ which satisfies the oddtown property. Let $\mathcal A_2$ be the maximum oddtown subfamily in $\mathcal A\setminus \mathcal A_1$. Then for any $i\ge 2$, as long as $\cup_{j\in [i]}\mathcal A_j\not=\mathcal A$, define $\mathcal A_{i+1}$ as the maximum oddtown subfamily in $\mathcal A\backslash \cup_{j\in [i]}\mathcal A_j$. This process will terminate when some $\mathcal A_r\neq\emptyset$ {satisfies} $\cup_{j\in [r]}\mathcal A_j=\mathcal A$. Trivially $r\le|\mathcal A|$. Since each $\mathcal A_i$ is an oddtown subfamily, $|\mathcal A_i|\le n$ and $r\ge {\lceil s\slash n\rceil}+1$.

Because of the maximality of each $\mathcal A_i$, for any $j\in [r]$ with $j>i$, there exists at least one edge from $A$ to $\mathcal A_i$ in $H(\mathcal A)$ for any $A\in\mathcal A_j$. This means in $H(\mathcal A)$, $e(\mathcal A_i, \mathcal A_j)\ge |\mathcal A_j|$. As a consequence,
\begin{equation*}
\begin{aligned}
op(\mathcal A)&\ge e(\mathcal A_1, \mathcal A\backslash\mathcal A_1)+ e(\mathcal A_2, \mathcal A\backslash(\mathcal A_1\cup \mathcal A_2))+\cdots +e(\mathcal A_{r-1}, \mathcal A\backslash \cup_{j\le r-1}\mathcal A_j)\\
&\ge |\mathcal A\backslash\mathcal A_1|+|\mathcal A\backslash(\mathcal A_1\cup \mathcal A_2)|+\cdots+ |\mathcal A\backslash \cup_{j\le r-1}\mathcal A_j|\\
&\ge s+ (s-n)+\cdots +(s-n\lfloor s\slash n\rfloor)\\
&=\binom{\lfloor s\slash n\rfloor+1}2 n+(\lfloor s\slash n\rfloor+1)(s-n\lfloor s\slash n\rfloor).
\end{aligned}
\end{equation*}
\end{proof}
\begin{proof}[Proof of Theorem~\ref{theorem_oddtown_approximation}]
When $s=cn+o(n)$ for some constant $c$,  


$$\binom{\lfloor s\slash n\rfloor+1}2 n+(\lfloor s\slash n\rfloor+1)(s-n\lfloor s\slash n\rfloor)=\binom{\lfloor c\rfloor+1}2 n+(\lfloor c\rfloor+1)(c-\lfloor c\rfloor)n+ o(n).$$
For the tightness, choose $m=2\lceil\log{(\lfloor c\rfloor+3)}\rceil$, which is a constant integer satisfying $2^{\lfloor m\slash2\rfloor}\ge\lfloor c\rfloor+3.$ By the extremal size of an eventown family, there exists an eventown subfamily $\mathcal E_1\subset 2^{[m]}$ with size $\lfloor c\rfloor+3.$ Consider the family $\mathcal A(\mathcal E_1, m, n)$ from Construction~\ref{cons_A_m_I}. By the definition of $m$, $ |\mathcal A(\mathcal E_1, m, n)|=(\lfloor c\rfloor+3)(n-m)\geq s+n$. Choose our family $\mathcal A$ as a subfamily of $\mathcal A(\mathcal E_1, m, n)$ with size $s+n$, such that
each element in $[m+1,n]$ appears almost equally often in $\mathcal A$. Note that $|\mathcal A(\mathcal E_1, m, n)|-|\mathcal A|= n+ (1+\lfloor c\rfloor-c)n+o(n)$ and $m$ is a small constant. This means when $c>\lfloor c\rfloor$, there are in total $(c-\lfloor c\rfloor)n-o(n)$ elements in $[m+1, n]$ each appearing  $\lfloor c\rfloor+2$ times in  $\mathcal A$ and $(1+\lfloor c\rfloor-c)n+o(n)$ elements in $ [m+1, n]$ each appearing  $\lfloor c\rfloor+1$ times in  $\mathcal A$.
{When $c=\lfloor c\rfloor$, there are $n-o(n)$ elements each appearing  $\lfloor c\rfloor+1$ times, and $o(n)$ elements each appearing $\lfloor c\rfloor+2$ times if $|\mathcal A(\mathcal E_1, m, n)|-|\mathcal A|<2(n-m)$, or $\lfloor c\rfloor$ times if $|\mathcal A(\mathcal E_1, m, n)|-|\mathcal A|>2(n-m)$.}
Since subsets in $\mathcal A$ only have odd intersections with the subsets in $\mathcal A$ sharing the same element in $[m+1, n]$, for both cases,
\begin{equation*}
\begin{aligned}
op(\mathcal A)&=\binom{\lfloor c\rfloor+2}2(c-\lfloor c\rfloor)n+ \binom{\lfloor c\rfloor+1}2(1+\lfloor c\rfloor-c)n+o(n)\\
&=\binom{\lfloor c\rfloor+1}2 n+(\lfloor c\rfloor+1)(c-\lfloor c\rfloor)n+ o(n).
\end{aligned}
\end{equation*}
\end{proof}
By using the same analysis as in the proof of Theorem~\ref{theorem_oddtown_approximation}, we can give more asymptotic results for $M_o(s, n)$ for much larger $s$ by choosing some proper $m$. We list some results here and omit the proofs for brevity.
\begin{itemize}
\item As long as $s=o(n^2\slash\log n)$ and $n=o(s)$, $M_o(s, n)=\frac12s^2\slash n+\frac12 s+o(s)$.
\item As long as $\log s=o(n)$ and $n=o(s)$, $M_o(s, n)=\frac12s^2\slash n+o(s^2\slash n)$.
\end{itemize}
As one can see, more restrictions on $s$ lead to a more accurate result.

\section{Supersaturation problem of eventown}\label{sec_supersaturation_eventown}

The first part of this section devotes to prove Theorem~\ref{thm_for_eventown}.
Motivated by the proof strategy in \cite{o2021towards}, for a given even-sized family $\mathcal A$ on $[n]$, we pay special attention to the maximum eventown subfamily $\mathcal A'$ of $\mathcal A$, i.e., the largest subfamily of $\mathcal A$ which satisfies eventown property. Then $\mathcal A'$ is an independent set of $H(\mathcal A)$ with maximum size.  Further, any independent set of $H(\mathcal A)$ is an eventown subfamily of $\mathcal A$, and vise versa. We first show that if $|\mathcal A'|$ is either too large or too small, $\mathcal A$ must have large odd pair number. For convenience, let $N:=2^{\lfloor n\slash 2\rfloor}+s$.
\begin{lemma}\label{lemma_size_o_max_subeventown}
Let $n$ be a positive integer and $s\in [2^{\lfloor n\slash2\rfloor}- 2^{\lfloor n\slash4\rfloor}]$ satisfying $\lfloor\frac n 2\rfloor> 2\log(s+1)$. Let $\mathcal A$ be an even-sized family on $[n]$ with $|\mathcal A|=N$ and $\mathcal A'$ be its maximum eventown subfamily. If $|\mathcal A'|\ge 2^{\lfloor n\slash 2\rfloor-1}+s$ or $|\mathcal A'|\le \lceil\frac N {s+2}\rceil$, then $op(\mathcal A)\ge s\cdot 2^{\lfloor n\slash 2\rfloor-1}$.
\end{lemma}
\begin{proof} Let $t=|\mathcal A'|$, and without loss of generality assume that $\mathcal A'=\{A_1, A_2, \ldots A_t\}$.

First we consider the case $t\ge 2^{\lfloor n\slash 2\rfloor-1}+s$. For each $A_i$, write $\bm v_i=\bm v_{A_i}$ for short, that is  the characteristic vector of $A_i$ in $\mathbb F_2^n$. Consider $W=span(\bm v_1, \bm v_2, \ldots, \bm v_t)$, which is a subspace of $\mathbb F_2^n$. Since $\mathcal A'$ is an eventown family on $[n]$, the inner product of any two vectors in $W$ is zero, and hence $W\subset W^\perp$. By the fact that $\dim(W)+\dim(W^\perp)=\dim(\mathbb F_2^n)=n$, we have $\dim(W)\le \lfloor n\slash 2\rfloor$. However, the subspace $W$ has at least $t> 2^{\lfloor n\slash 2\rfloor-1}$  different vectors, so $\dim(W)= \lfloor n\slash 2\rfloor$.

Note that $\emptyset\not\in\mathcal A\backslash \mathcal A'$, otherwise  $\mathcal A'\cup \{\emptyset\}$ is a larger eventown subfamily of $\mathcal A$.
So for any $A\in \mathcal A\backslash \mathcal A'$,  its characteristic vector $\bm v_A$ is not zero. For any $A_i\in \mathcal A'$, $|A\cap A_i|$ is odd if and only if $(\bm v_A, \bm v_i)=1$. Since $\mathcal A'\cup\{A\}$ is no longer an eventown family, $W\not\subset \bm v_A^\perp$ and hence $\dim(W\cap\bm v_A^\perp)=\dim(W)-1$. Let $N(A)$ denote the number of elements in $\mathcal A$ which has odd-sized intersection with $A$. Then $|N(A)|\ge |\mathcal A'|-|W\cap \bm v_A^\perp| {\ge} t-2^{\lfloor n\slash 2\rfloor-1}$, and 
\begin{equation}
op(\mathcal A)\ge \sum_{A\in\mathcal A\backslash \mathcal A'}|N(A)|\geq(N-t)(t-2^{\lfloor n\slash 2\rfloor-1}).
\end{equation}
Since $t\ge 2^{\lfloor n\slash 2\rfloor-1}+s$ and $t\le 2^{\lfloor n\slash 2\rfloor}$ (the latter is from that $\mathcal A'$ is an eventown family), $op(\mathcal A)\ge s\cdot 2^{\lfloor n\slash 2\rfloor-1}$.

Next we consider the case $t\le \lceil\frac N {s+2}\rceil:=\alpha$. Since $\lfloor\frac n 2\rfloor> 2\log(s+1)$, we have $(s+1)\alpha<N \leq (s+2)\alpha$. Then $\lceil\frac N \alpha\rceil= s+2$.

Similar to the process in the proof of Lemma~\ref{lemma_odd_approx_lower}, we define a sequence of disjoint subcollections of $\mathcal A$: $\mathcal A_1, \mathcal A_2, \ldots$ recursively. First let $\mathcal A_1=\mathcal A'$, and let $\mathcal A_2$ be the maximum eventown subfamily in $\mathcal A\setminus \mathcal A'$. Then for any $i\ge 2$, as long as $\cup_{j\in [i]}\mathcal A_j\not=\mathcal A$, define $\mathcal A_{i+1}$ as the maximum eventown subfamily in $\mathcal A\backslash \cup_{j\in [i]}\mathcal A_j$. This process will terminate when some $\mathcal A_r\neq\emptyset$ {satisfies} $\cup_{j\in [r]}\mathcal A_j=\mathcal A$. Trivially $r\le|\mathcal A|$. Since each $\mathcal A_i, i\ge 1$ is also an eventown subfamily in $\mathcal A$, then $|\mathcal A_i|\le |\mathcal A'|\le \alpha$ for any $i\ge 1$. So $r\ge \lceil\frac N \alpha\rceil=s+2$.

Consider the odd pair graph $H(\mathcal A)$. Since $\mathcal A_1$ is  a maximal independent set of $H(\mathcal A)$, each vertex in $\mathcal A\backslash\mathcal A_1$ has at least one neighbour in $\mathcal A_1$. So $e(\mathcal A_1, \mathcal A\backslash\mathcal A_1)\ge |\mathcal A\backslash\mathcal A_1|\ge N-\alpha$. Similarly, for any $i\in [s+1]$,  $\mathcal A_i$ is a maximal independent set of $\mathcal A\backslash \cup_{j\le i-1}\mathcal A_j$. Then each vertex in $\mathcal A\backslash \cup_{j\le i}\mathcal A_j$ has at least one neighbour in $\mathcal A_i$. So $e(\mathcal A_i, \mathcal A\backslash \cup_{j\le i}\mathcal A_j)\ge |\mathcal A\backslash \cup_{j\le i}\mathcal A_j|\ge N-i\cdot \alpha$. As a consequence, the number of edges in $H(\mathcal A)$ is
\begin{equation*}
\begin{aligned}
op(\mathcal A)&\ge e(\mathcal A_1, \mathcal A\backslash\mathcal A_1)+ e(\mathcal A_2, \mathcal A\backslash(\mathcal A_1\cup \mathcal A_2))+\cdots +e(\mathcal A_{r-1}, \mathcal A\backslash \cup_{j\le r-1}\mathcal A_j)\\
&\ge |\mathcal A\backslash\mathcal A_1|+|\mathcal A\backslash(\mathcal A_1\cup \mathcal A_2)|+\cdots+ |\mathcal A\backslash \cup_{j\le r-1}\mathcal A_j|\\
&\ge (N-\alpha)+ (N-2\alpha)+\cdots +(N-(s+1)\alpha)\\
&=(N-\frac{(s+2)\alpha}2)(s+1).
\end{aligned}
\end{equation*}
Since  $(s+2)\alpha=\lceil\frac N\alpha\rceil\alpha\le N+\alpha$, we have $op(\mathcal A)\ge (N-\frac{N+\alpha}2)(s+1)=\frac{s+1}2(N-\alpha)$. If $\frac{s+1}2(N-\alpha)\ge s\cdot 2^{\lfloor\frac n 2\rfloor-1}$, then our proof is finished. In fact, by doubling both sides and computing their difference, we have
\begin{equation*}
\begin{aligned}
(s+1)(N-\alpha)-s\cdot2^{\lfloor\frac n 2\rfloor}&=s(2^{\lfloor\frac n 2\rfloor}+s-2^{\lfloor\frac n 2\rfloor}-\alpha)+N-\alpha\\
&=s^2+N-(s+1)\alpha,
\end{aligned}
\end{equation*}
which is positive since $(s+1)\alpha<N$.
\end{proof}

To complete the proof of Theorem~\ref{thm_for_eventown}, we are left to check the case when $|\mathcal A'|$ is in the range $ [\lceil\frac N{s+2}\rceil+1, 2^{\lfloor n\slash 2\rfloor -1}+s-1]$, for which we have the following lemma.

\begin{lemma}\label{lemma_middlesize}
Let $n$ be a sufficiently large integer and $s\in [2^{\lfloor\frac n 8\rfloor}\slash n]$. Let $\mathcal A$ be an even-sized family on $[n]$ with $|\mathcal A|=N$ and $\mathcal A'$ be its maximum eventown subfamily. If $|\mathcal A'|\in [\lceil\frac N{s+2}\rceil+1, 2^{\lfloor n\slash 2\rfloor -1}+s-1]$ and  $op(\mathcal A)< s\cdot 2^{\lfloor n\slash 2\rfloor -1}$, then for any  $A\in\mathcal A\backslash \mathcal A'$, $op(\mathcal A'\cup\{A\})\ge s+1.$
\end{lemma}

Lemma~\ref{lemma_middlesize} trivially leads to a contradiction by 
\begin{equation*}
op(\mathcal A)> |\mathcal A\backslash\mathcal A'|\cdot s\ge (2^{\lfloor\frac n 2\rfloor} +s-(2^{\lfloor n\slash 2\rfloor -1}+s-1))\cdot s> s\cdot 2^{\lfloor n\slash 2\rfloor -1},
\end{equation*}and thus completes the proof of Theorem~\ref{thm_for_eventown}.
However the proof of Lemma~\ref{lemma_middlesize} is more involved and far from trivial, and we defer it to the next subsection.




\subsection{Proof of Lemma~\ref{lemma_middlesize}}

In this subsection, we always assume that the conditions in Lemma~\ref{lemma_middlesize} are all satisfied. 
Remember that we use $H_0$ to denote the odd pair graph of $2^{[n]}$. Consider the induced bipartite graph $H_0[\mathcal A', X]$ with $X\triangleq \mathcal A\backslash\mathcal A'$. It is easy to check that $H_0[\mathcal A', X]$ is a subgraph of $H(\mathcal A)$.
Note that any vertex $x$ in  $X$ has at least one neighbor in $\mathcal A'$, otherwise  $\mathcal A'\cup\{x\}$ is a larger eventown subfamily. Under these notations,  Lemma~\ref{lemma_middlesize} is equivalent to saying that $d_{H_0[\mathcal A', X]}(x)\ge s+1$ for any $x\in X$. Suppose on the contrary that there exists a vertex $x\in X$ with degree at most $s$ in $H_0[\mathcal A', X]$. Denote its neighborhood in $H_0[\mathcal A', X]$ as $N(x)\subset \mathcal A'$. Further denote $Y$ as the set of vertices in $X$ whose neighborhood in $H_0[\mathcal A', X]$ is contained in $N(x)$, i.e., $Y=\{y\in X: N(y)\subseteq N(x)\}$. As $x\in Y$, $Y$ is not empty. We claim that the size of $Y$ is very small.
\begin{claim}\label{claim_Y_is_small}
The size $|Y|<s\cdot (2^{n\slash 4}+1)$.
\end{claim}
\begin{proof} Denote $\ell=|Y|$. We first claim that $Y$ does not contain an independent set of ${H_0[X]}$ of size $s+1$. Otherwise, say $I\subset Y$ is an independent set with $|I|=s+1$. Then ${\mathcal B}=\big(\mathcal A'\backslash N(x)\big)\cup I$ is also an independent set, i.e., an eventown subfamily of ${\mathcal A}$, since $N(I)\subseteq N(Y)\subseteq N(x)$. Since $|N(x)|\le s$, ${\mathcal B}$ is of size $|\mathcal A'|-|N(x)|+|I|\ge |\mathcal A'|+1$, which contradicts the maximality of $|\mathcal A'|$.


Consider the induced subgraph $H(Y)=H(\mathcal A)[Y]=H_0[Y]$, which does not contain an independent set of size $s+1$ by the above analysis. Equivalently, the complement  of $H(Y)$, that is $\overline {H}(Y)$, does not contain any copy of $K_{s+1}$. By Theorem~\ref{coro_turan},
\begin{equation*}
e(\overline {H}(Y))\le \frac{\ell^2}2\cdot\frac{s-1}s.
\end{equation*}
So $e(H(Y))\ge \binom{\ell}2-{\ell^2(s-1)}\slash 2s={\ell^2}\slash{2s}- {\ell}\slash2$. Since $e(H(Y))\le e(H(\mathcal A))=op(\mathcal A)< s\cdot 2^{\lfloor\frac n 2\rfloor-1}$, then ${\ell^2}\slash{2s}- {\ell}\slash2< s\cdot 2^{\lfloor\frac n 2\rfloor-1}$, which leads to our desired result.
\end{proof}
\begin{remark}
Note that in our analysis $|\mathcal A'|<2^{\lfloor\frac n 2\rfloor-1}+s$, so $|X|=N-|\mathcal A'|>2^{\lfloor\frac n 2\rfloor-1}$, and hence $|Y|=o(|X|)$. This means $|X\setminus Y|$ has the same order as $|X|$.
\end{remark}

Since $N(x)\neq \emptyset \neq Y$, we consider the induced subgraph $H_0[\mathcal A'\backslash N(x), X\backslash Y]$. Note that any vertex in $X\backslash Y$ has at least one neighbor in $\mathcal A'\backslash N(x)$ by the definition of $Y$. The following claim shows that after deleting $N(x)$ and $Y$, the  part $X\backslash Y$ still cannot reach the minimum degree $s+1$ in $H_0[\mathcal A'\backslash N(x), X\backslash Y]$.
\begin{claim}\label{claim_smallest_degree_is_small_after_del}
There exists a vertex $v\in X\backslash Y$ with degree at most $ s$ on $H_0[\mathcal A'\backslash N(x), X\backslash Y]$.
\end{claim}
\begin{proof}
We prove a stronger claim that there exists a vertex $v\in X\backslash Y$ with degree at most $ s$ on $H_0[\mathcal A', X]$. Suppose on the contrary, every vertex in $X\backslash Y$ has degree at least $s+1$. Then $op(\mathcal A)\ge e(\mathcal A', X\backslash Y)\ge (|X|-|Y|)(s+1)$. Since $|X|>2^{\lfloor\frac n 2\rfloor-1}$ and $|Y|<s\cdot (2^{n\slash 4}+1)$, then
$$
\begin{aligned}
op(\mathcal A)-s\cdot 2^{\lfloor\frac n 2\rfloor-1}&\ge (2^{\lfloor\frac n 2\rfloor-1}+1-s(2^{\frac n 4}+1))(s+1)- s\cdot 2^{\lfloor\frac n 2\rfloor-1}\\
&{=2^{\lfloor\frac n 2\rfloor-1}-s(s+1)(2^{\frac n 4}+1)+s+1.}
\end{aligned}
$$
Since $s\le 2^{\lfloor\frac n 8\rfloor}\slash n$, ${s(s+1)(2^{\frac n 4}+1)}=O(2^{\frac n 2}\slash n^2)=o(2^{\lfloor\frac n 2\rfloor-1})$. So $op(\mathcal A)>s\cdot 2^{\lfloor\frac n 2\rfloor-1}$, which contradicts the assumption in Lemma~\ref{lemma_middlesize}.
\end{proof}

By Claim~\ref{claim_smallest_degree_is_small_after_del}, we can see the graph $H_0[\mathcal A'\backslash N(x), X\backslash Y]$ has the same property as $H_0[\mathcal A', X]$, i.e., the part $X\setminus Y$ has no isolated vertex but has a small degree ($\leq s$) vertex. Motivated by this, we can set $\mathcal A'\backslash N(x)$ and $X\backslash Y$ as new $\mathcal A'$ and $X$, and do the same analysis as Claims~\ref{claim_Y_is_small} and~\ref{claim_smallest_degree_is_small_after_del} iteratively. The detail of the induction is as follows, where we denote $W_{\mathcal B}$ the subspace $span(\bm v_B:B\in \mathcal B)\subset \mathbb F_2^n$ for any $\mathcal B\subset \mathcal A$.

For a general index  $i\geq 1$, we define the following conditions and notations for a pair $(\mathcal A_i, X_i)$:
\begin{itemize}
  \item[ (C1)]$\mathcal A_i\subset \mathcal A'$ and $|\mathcal A_i|\geq |\mathcal A'|-(i-1)s$;
      \item[ (C2)]$\dim(W_{\mathcal A_i})= \dim(W_{\mathcal A'})+1-i$;
  \item[ (C3)]$X_i\subset X$ and  $|X_i|\ge |X|- \binom{i}{2} s(2^{\frac n 4}+1)$; and
  \item[ (C4)]in $H_i:=H_0[\mathcal A_i,X_i]$, there is a vertex $x_i\in X_i$ such that $1\leq d_{H_i}(x_i)\leq s$.
\end{itemize} Then based on the chosen $x_i$, define  $N_i:=N_{H_i}(x_i) \subset \mathcal A_i$, which satisfies $1\leq |N_i|\leq s$ by (C4), and define $Y_i:=\{y\in X_i: N_{H_i}(y)\subset N_i\}\subset X_i$, which is nonempty since $x_i\in Y_i$. For the index $i+1$, set
 $\mathcal A_{i+1}:=\mathcal A_i\setminus N_i$, $X_{i+1}:=X_i\setminus Y_i$ and $H_{i+1}:=H_0[\mathcal A_{i+1},X_{i+1}]$.

Notice that for large $i$, the existence of a pair $(\mathcal A_i, X_i)$ satisfying (C1)-(C4) itself can lead to a contradiction. To see this, we mention that the lower bound of $|\mathcal A_i|$ and the upper bound of $\dim(W_{\mathcal A_i})$ are given in (C1) and (C2), respectively, but $\log |\mathcal A_i|$ can never exceed $\dim(W_{\mathcal A_i})$. As a consequence, our strategy is to prove the existence of a pair of $(\mathcal A_i, X_i)$ iteratively from $i=1$ to some $i$ large enough to trigger such a contradiction.

 As the base case, we set
$\mathcal A_1:=\mathcal A'$, $X_1:=X$ and $H_1:=H_0[\mathcal A_1,X_1]$. There exists a vertex $x_1:=x\in X_1$ with degree less than $s+1$ in $H_1$ by the original assumption,  and $ d_{H_1}(x_1)\geq 1$ by the maximality of $\mathcal A'$.  So conditions (C1)-(C4) trivially hold for $i=1$.

 Now we check the conditions for $i=2$. Define  $N_1:=N(x)=N_{H_1}(x_1)$, which is not empty and has size at most $s$, and define $Y_1:=Y=\{y\in X_1: N_{H_1}(y)\subset N_1\}$, which has size at most $s\cdot (2^{n\slash 4}+1)$ by Claim~\ref{claim_Y_is_small}. Then (C1) and (C3) hold for $i=2$. Since $ d_{H_1}(x_1)\geq 1$, that is, $\mathcal A' \not\subset x_1^\perp$, we have $\dim (W_{\mathcal A'} \cap x_1^\perp)=\dim(W_{\mathcal A'})-1$. Note that $\mathcal A_2=\mathcal A'\cap x_1^\perp$, so $\dim(W_{\mathcal A_2})=\dim(W_{\mathcal A'})-1$ and (C2) holds for $i=2$. Finally, (C4) holds for $i=2$ by Claim~\ref{claim_smallest_degree_is_small_after_del}, and the nonzero degree is from the definition of $Y_1$.

 In general, we assume that (C1)-(C4) hold for $i$. We will show that as long as both $\mathcal A_{i+1}$ and $X_{i+1}$ are nonempty and $i\leq n\slash 2$, all conditions (C1)-(C4) still hold for $i+1$.

 First, since $|N_i|=d_{H_i}(x_i)\leq s$ by (C4) and $|\mathcal A_i|\geq |\mathcal A'|-(i-1)s$ by (C1), we have  $|\mathcal A_{i+1}|= |\mathcal A_{i}|-|N_i|\geq |\mathcal A'|-(i-1)s -s=|\mathcal A'|-is$, i.e., (C1) holds for $i+1$. Second, since $d_{H_i}(x_i)\geq 1$ by (C4), that is $\mathcal A_i\not\subset x_i^\perp$, then $\dim(W_{\mathcal A_{i+1}})=\dim (W_{\mathcal A_i} \cap x_i^\perp)=\dim(W_{\mathcal A_i})-1= \dim (W_{\mathcal A'})-i$ by (C2). So we have proved (C2) for $i+1$.
 Third, to prove (C3) for $i+1$, we need to show the following result as in Claim~\ref{claim_Y_is_small}.

\begin{claim}\label{claim_Y_i_is_small}
The size $|Y_i|<i s\cdot (2^{n\slash 4}+1)$.

\end{claim}
\begin{proof}The proof is similar to that of Claim~\ref{claim_Y_is_small}.

Denote $\ell_i=|Y_i|$. We first claim that $H_0[Y_i]$ does not have an independent set of size $is+1$. Otherwise, say $I_i\subset Y_i$ is an independent set with $|I_i|=is+1$. Then ${\mathcal B_i}=\big(\mathcal A_i\backslash N_i\big)\cup I_i$ is also an independent set in $H(\mathcal A)$, i.e., an eventown subfamily of ${\mathcal A}$, since $N_{H_1}(I_i)\subset N_{H_1}(Y_i) \subset N_i \cup ({\mathcal A}\setminus \mathcal A_i)$. Since $|N_i|\le s$, ${\mathcal B_i}$ is of size $|\mathcal A_i|-|N_i|+|I_i|\ge |\mathcal A'|-(i-1)s-s+is+1\geq |\mathcal A'|+1$ by (C1), which contradicts the maximality of $|\mathcal A'|$.

Then the induced subgraph $H_0[Y_i]$ does not contain an independent set of size $is+1$ by the above analysis. That is, the complement  $\overline {H_0}[Y_i]$ does not contain any copy of $K_{is+1}$. By Corollary~\ref{coro_turan},
\begin{equation*}
e(\overline {H_0}[Y_i])\le \frac{\ell_i^2}2\cdot\frac{is-1}{is}.
\end{equation*}
So $e(H_0[Y_i])\ge \binom{\ell_i}2-{\ell_i^2(is-1)}\slash 2is={\ell_i^2}\slash{2is}- {\ell_i}\slash2$. Since $e(H_0[Y_i])\le e(H(\mathcal A))=op(\mathcal A)< s\cdot 2^{\lfloor\frac n 2\rfloor-1}$, then ${\ell_i^2}\slash{2is}- {\ell_i}\slash2< s\cdot 2^{\lfloor\frac n 2\rfloor-1}$, which leads to $\ell_i<is(2^{\frac n 4}+1)$.
%
%
%
\end{proof}

By Claim~\ref{claim_Y_i_is_small}, $|X_{i+1}|=|X_i|-|Y_i|\geq |X|-\binom{i}{2} s(2^{\frac n 4}+1)-i s\cdot (2^{n\slash 4}+1)=|X|-\binom{i+1}{2} s(2^{\frac n 4}+1)$. So (C3) holds for $i+1$. It is left to prove (C4) for $i+1$. The nonzero degree of each vertex in $X_{i+1}$ in $H_{i+1}$ is from the definition of $Y_i$, so we only need to prove a result similar to  Claim~\ref{claim_smallest_degree_is_small_after_del}.

\begin{claim}\label{claim_small_degree_for_all_s}
If $i\leq n\slash 2$, then there exists a vertex $v\in X_{i+1}$ with degree at most $ s$ on $H_{i+1}$.
\end{claim}
\begin{proof} As in the proof of Claim~\ref{claim_smallest_degree_is_small_after_del}, we assume that each vertex $v\in X_{i+1}$ has degree at least $s+1$ in the bigger graph $H_i$. Then $op(\mathcal A)\ge e(\mathcal A_i, X_{i+1})\ge |X_{i+1}|(s+1)$. Since $|X_{i+1}|\ge |X|-\binom{i+1}{2}s(2^{\frac n 4}+1)$ by (C3) for $i+1$,
$$
\begin{aligned}
& op(\mathcal A)-s\cdot2^{\lfloor\frac n 2\rfloor-1}\\
\geq & \left(|X|-\binom{i+1}{2}s(2^{\frac n 4}+1)\right)(s+1)-s\cdot2^{\lfloor\frac n 2\rfloor-1}\\
\ge & \left(2^{\lfloor\frac n 2\rfloor-1}+1-\frac{i(i+1)}2(2^{\frac n 4}+1)s\right)(s+1)-s\cdot2^{\lfloor\frac n 2\rfloor-1}\\
=& s-\frac{i(i+1)}2 s^2(2^{\frac n 4}+1) + 2^{\lfloor\frac n 2\rfloor-1}+1-\frac{i(i+1)}2(2^{\frac n 4}+1)s\\
=& 2^{\lfloor\frac n 2\rfloor-1}-\frac{i(i+1)} 2 s(s+1)(2^{\frac n 4}+1)+s+1.
\end{aligned}
$$
Since $i\le n\slash 2$ and $s\le 2^{\lfloor\frac n 8\rfloor}\slash n$,
$$
\begin{aligned}
& \frac{i(i+1)} 2 s(s+1)(2^{\frac n 4}+1)\\
\le& \frac{n(n+2)} 8 s(s+1)(2^{\frac n 4}+1)\\
=&\frac 1 8 (2^{\lfloor\frac n 8\rfloor})(2^{\lfloor\frac n 8\rfloor}+o(2^{\frac n 8}))(2^{\frac n 4}+1)\\
\le & 2^{\frac n 2-3}+o(2^\frac n 2)< 2^{\lfloor\frac n 2\rfloor-1}.
\end{aligned}
$$
So $op(\mathcal A)>s\cdot2^{\lfloor\frac n 2\rfloor-1}$, contradicting to the hypothesis of Lemma~\ref{lemma_middlesize} again.
\end{proof}

By Claim~\ref{claim_small_degree_for_all_s}, (C4) holds for $i+1$. Thus we have proved that all conditions (C1)-(C4)  hold for $i+1$ if $i\leq n/2$.
So we can continue this induction until some $i<n\slash 2$ such that either $\mathcal A_i$ or $X_i$ is empty, or, until $i=\lfloor n\slash2\rfloor$. Finally, we claim that both $\mathcal A_{\lfloor n\slash2\rfloor}$ and $X_{\lfloor n\slash2\rfloor}$ are of large sizes (larger than any given constant),
but $\dim(W_{\mathcal A_{\lfloor n\slash2\rfloor}})=\dim(W_{\mathcal A'})+1-\lfloor n\slash2\rfloor\leq \lfloor n\slash2\rfloor+1-\lfloor n\slash2\rfloor=1$ by (C2), which is a contradiction. This means the original assumption ``there exists a vertex $x\in X$ with degree at most $s$ in $H_0[\mathcal A', X]$'' should not happen at the beginning, so Lemma~\ref{lemma_middlesize} is proved.

Now we show that $|\mathcal A_{\lfloor n\slash2\rfloor}|$ and $|X_{\lfloor n\slash2\rfloor}|$ are both large. 
Remember that $N=2^{\lfloor n\slash 2\rfloor}+s$, $s\le 2^{\lfloor\frac n 8\rfloor}\slash n$, and $\frac N {s+2}\le \lceil\frac N {s+2}\rceil <|\mathcal A'|< 2^{\lfloor\frac n 2\rfloor-1}+s$. By (C1), $|\mathcal A_{\lfloor n\slash2\rfloor}|\ge |\mathcal A'|-ns\slash 2> \frac N {s+2}-ns\slash 2$. This leads to $|\mathcal A_{\lfloor n\slash2\rfloor}|\ge \frac N {s+2}(1-o(1))>4$, since $$\frac N {s+2}\ge\frac{2^{\lfloor\frac n 2\rfloor}+s}{2^{\lfloor\frac n 8\rfloor}\slash n+2}=\Theta(n\cdot 2^{\frac 3 8 n}), \text{ but } ns=O(2^\frac n 8).$$
By (C3) we have $|X_{\lfloor n\slash2\rfloor}|\ge |X|-n^2s(2^{\frac n 4}+1)$. Since  $|X|= |\mathcal A|- |\mathcal A'| > 2^{\lfloor\frac n 2\rfloor-1}$, $|X_{\lfloor n\slash2\rfloor}|\ge 2^{\lfloor\frac n 2\rfloor-1} -O(n\cdot 2^{\frac 3 8 n})=2^{\lfloor\frac n 2\rfloor-1}(1-o(1))$. So both $|\mathcal A_{\lfloor n\slash2\rfloor}|$ and $|X_{\lfloor n\slash2\rfloor}|$ are large when $n$ is large.


\subsection{Asymptotic result}
This subsection is devoted to prove Theorem~\ref{thm_eventownlar} and Theorem~\ref{thm_complete_old_eventown} by using Fourier analysis.
Consider the characters on additive abelian group $\mathbb F_2^n$. For any $\bm m\in \mathbb F_2^n$, define $\chi_{\bm m}: \mathbb F_2^n\to\mathbb C^*$ as $\chi_{\bm m}(a)=e^{\pi i(\bm m, \bm a)}=(-1)^{(\bm m, \bm a)}$ for any $\bm a\in \mathbb F_2^n$. It is easy to check that for all $\bm a, \bm b\in \mathbb F_2^n$, $\chi_{\bm m}(\bm a+\bm b)=\chi_{\bm m}(\bm a)\chi_{\bm m}(\bm b)$, so $\chi_{\bm m}\in\widehat{\mathbb F_2^n}$.
Moreover, for any $\bm m\ne \bm m'\in \mathbb F_2^n$, we have $\chi_{\bm m}\ne\chi_{\bm m'}$. The reason is that there always exists some $\bm a\in \mathbb F_2^n$ such that $(\bm a, \bm m-\bm m')=1$, which leads to $\chi_{\bm m}(\bm a)=(-1)^{(\bm m, \bm a)}\ne (-1)^{(\bm m', \bm a)}= \chi_{\bm m'}(\bm a).$ From $\{\chi_{\bm m}: \bm m\in \mathbb F_2^n\}\subseteq\widehat {\mathbb F_2^n}$ and $|\{\chi_{\bm m}: \bm m\in \mathbb F_2^n\}|=|\mathbb F_2^n|=|\widehat {\mathbb F_2^n}|$, we have $\widehat {\mathbb F_2^n}= \{\chi_{\bm m}: \bm m\in \mathbb F_2^n\}$.

For any subfamily $\mathcal A\subset 2^{[n]}$, consider its odd pair graph $H=H(\mathcal A)$, which is an induced subgraph of $H_0$. Denote $V(H)$ as the vertex set of $H$.
We have the following concentration result for the edge number of $H$.

\begin{lemma}\label{lemma_concentration_result}
Let $H=H(\mathcal A)$ for some subfamily $\mathcal A\subset 2^{[n]}$. Let $v_o(H)$ be the number of vertices in $H$ which are odd-sized  elements in $\mathcal A$. Then we have
$$|v(H)^2-4e(H)-2v_o(H)|\le 2^{n\slash 2}v(H).$$
In particular, when $\mathcal A$ is even-sized,
$$|v(H)^2-4e(H)|\le 2^{n\slash 2}v(H).$$
\end{lemma}
\begin{proof}
Define $f: \mathbb F_2^n\to\mathbb C$ as the indicator function of $V(H)$. In other words, $f(\bm a)=1$ if $\bm a$ is a vertex of $H$, and $f(\bm a)=0$ otherwise. Then for any $\bm m\in V(H)$,
$$
\begin{aligned}
\hat f(\chi_{\bm m})&=\frac 1{|\mathbb F_2^n|}\sum_{\bm a\in \mathbb F_2^n}f(\bm a)\overline{\chi_{\bm m}(\bm a)}\\
&=\frac 1 {2^n}\sum_{\bm a\in\mathbb F_2^n} f(\bm a)(-1)^{(\bm m, \bm a)}\\
&=\frac 1 {2^n}\sum_{\bm a\in V(H)} (-1)^{(\bm m, \bm a)}.
\end{aligned}
$$
For $\bm a\ne \bm m$, $(\bm a, \bm m)=1$ if and only if $\{\bm a, \bm m\}$ forms an edge in $H$. For $\bm a=\bm m$, $(\bm m, \bm m)=1$ if and only if $(\bm 1, \bm m)=1$, where $\bm 1$ means the all-one vector. As a result, if $\bm m$ is odd-sized, $\hat f(\chi_{\bm m})= \frac 1 {2^n}[(-1)(d_{H}(\bm m)+1)+ 1\cdot(v(H)-d_{H}(\bm m)-1)]= \frac 1 {2^n}(v(H)-2d_{H}(\bm m)-2)$; if $\bm m$ is even-sized, $\hat f(\chi_{\bm m})= \frac 1 {2^n}[(-1)d_{H}(\bm m)+ 1\cdot(v(H)-d_{H}(\bm m))]= \frac 1 {2^n}(v(H)-2d_{H}(\bm m))$. To sum up, for any $\bm m\in v(H)$, \[\hat f(\chi_{\bm m})= \frac 1 {2^n}(v(H)-2d_{H}(\bm m)-2(\bm 1, \bm m)).\]

By Plancherel's formula,
$$\frac 1 {2^n}v(H)=\Vert f\Vert^2=\sum_{\bm m\in\mathbb F_2^n}|\hat f(\chi_{\bm m})|^2\ge \sum_{\bm m\in V(H)}|\hat f(\chi_{\bm m})|^2.$$
By Cauchy-Schwarz inequality,
\begin{equation*}
		 \begin{split}
		 \sum_{\bm m\in V(H)} \left| \hat{f}(\chi_{\bm m})\right|^2 &= \sum_{\bm m\in V(H)} \left( \frac{1}{2^n} \big( v(H)-2d_{H}(\bm m)-2(\bm 1, \bm m)\big) \right)^2 \\
		                &= \frac{1}{2^{2n}} \sum_{\bm m\in V(H)} \big( v(H)-2d_{H}(\bm m)-2(\bm 1, \bm m)\big)^2\\
		                &\ge \frac{1}{2^{2n}} \cdot \frac{1}{v(H)} \left( \sum_{\bm m\in V(H)} \big( v(H)-2d_{H}(\bm m)-2(\bm 1, \bm m)\big) \right)^2\\
		                &= \frac{1}{2^{2n}} \cdot \frac{1}{v(H)} \left( v(H)^2-4e(H)-2v_o(H) \right)^2.
		 \end{split}
		 \end{equation*}
Combining both inequalities above, $$\frac1{2^n}v(H)\ge \frac1{2^{2n}}\frac1{v(H)}(v(H)^2-4e(H)-2v_o(H))^2,$$ and hence $|v(H)^2-4e(H)-2v_o(H)|\le 2^\frac n 2v(H)$.
\end{proof}

$~$

\begin{proof}[Proof of Theorem~\ref{thm_eventownlar}]
Since $n\ge 1\slash\epsilon$, $2^{(1-\epsilon)n}\le 2^{n-1}$. There exists some even-sized $\mathcal A\subset 2^{[n]}$ such that $|\mathcal A|\ge 2^{(1-\epsilon)n}$.
Let $H=H(\mathcal A)$ be the odd pair graph of any  such $\mathcal A$. 
From Lemma~\ref{lemma_concentration_result}, we have
$${e(H)}\slash{\binom {v(H)} 2}\ge \frac{v(H)^2-2^{n\slash 2}v(H)}{2v(H)(v(H)-1)}\ge \frac12(1-2^{n\slash2}v(H)^{-1})\ge \frac 1 2(1-2^{(\epsilon-1\slash2)n}).$$
Hence, $f_n(\epsilon)=\min\{\frac{e(H(\mathcal A))}{\binom{|\mathcal A|} 2}: \text{ even-sized }\mathcal A\subset2^{[n]}; |\mathcal A|\ge 2^{(1-\epsilon)n}\}\ge \frac 1 2(1-2^{(\epsilon-1\slash2)n}).$

For the upper bound, let $\mathcal A$ be the collection of all even-sized sets from $2^{[n]}$. Then $|\mathcal A|=2^{n-1}\ge 2^{(1-\epsilon)n}$, and $V(H)$ forms a subspace of $\mathbb F_2^n$.
For any $\bm v\in V(H)$, it is clear that $d_H(\bm v)=0$ if and only if one of the following two cases happens:
\begin{itemize}
\item[(1)] $\bm v=\bm 0$;
\item[(2)] $\bm v=\bm 1$ and $n$ is even.
\end{itemize}
 Otherwise, $d_H(\bm v)=v(H)-|V(H)\cap {\bm v}^\perp|= 2^{n-1}-2^{n-2}=2^{n-2}$.

Thus, when $n$ is odd, $e(H)=\frac12\sum_{\bm v\in V(H)\backslash\{\bm 0\}}2^{n-2}=\frac12\binom{v(H)}2$; when $n$ is even, $e(H)=\frac12\sum_{\bm v\in V(H)\backslash\{\bm 0, \bm 1\}}2^{n-2}=v(H)(v(H)-2)\slash4$.
\end{proof}

$~$

When $n$ is odd, we can further improve Lemma~\ref{lemma_concentration_result} by considering the Fourier analysis on finite additive abelian subgroup $\bm 1^\perp\subset \mathbb F_2^n$. Note that $\bm 1^\perp$ is a subspace of $\mathbb F_2^n$ with dimension $n-1$ consisting of all vectors with even numbers of $1$, and hence the additive abelian subgroup structure follows naturally from the subspace structure.  For any $\bm m\in\bm 1^\perp$, consider $\chi_{\bm m}'$ as the restriction of $\chi_{\bm m}$ under $\bm 1^\perp$, i.e., it maps $\bm a$ to $\chi_{\bm m}(\bm a)=(-1)^{(\bm m, \bm a)}$ for any $\bm a\in \bm 1^\perp$. From $\chi_{\bm m}\in\widehat{\mathbb F_2^n}$, $\chi_{\bm m}'\in\widehat{\bm 1^\perp}$. Moreover, we claim that $\chi_{\bm m_1}'= \chi_{\bm m_2}'$ if and only if $\bm m_1=\bm m_2$ with $\bm m_1, \bm m_2\in \bm 1^\perp$. In fact, the equality $\chi_{\bm m_1}'= \chi_{\bm m_2}'$ means $(\bm m_1-\bm m_2, \bm a)=(\bm m_1, \bm a)-(\bm m_2, \bm a)=0$ for any $\bm a\in\bm 1^\perp$, which means $\bm m_1-\bm m_2\in(\bm 1^\perp)^\perp=span\{\bm 1\}$. If $\bm m_1\ne\bm m_2$, the only choice is $\bm m_1-\bm m_2=\bm 1$, which contradicts to both $\bm m_1, \bm m_2\in\bm 1^\perp$ for $n$ odd. Hence from the same analysis as in $ \widehat{\mathbb F_2^n}$, $\widehat {\bm 1^\perp}= \{\chi_{\bm m}': \bm m\in \bm 1^\perp\}$.

\begin{lemma}\label{lemma_concentration_result_odd}
Let $H=H(\mathcal A)$ for some subfamily $\mathcal A\subset 2^{[n]}$. When $n$ is odd and $\mathcal A$ is even-sized, we  have
$$|v(H)^2-4e(H)|\le 2^{(n-1)\slash 2}v(H).$$
\end{lemma}
\begin{proof}
Define $g: \bm 1^\perp\to\mathbb C$ as the indicator function of $V(H)$. Since $\mathcal A$ is even-sized, $V(H)\subseteq\bm 1^\perp$. For any $\bm m\in V(H)$,
$$
\begin{aligned}
\hat g(\chi_{\bm m}')&=\frac 1{|\bm 1^\perp|}\sum_{\bm a\in \bm 1^\perp}g(\bm a)\overline{\chi_{\bm m}'(\bm a)}\\
&=\frac 1 {2^{n-1}}\sum_{\bm a\in V(H)} (-1)^{(\bm m, \bm a)}\\
&=\frac 1 {2^{n-1}}[(-1)d_{H}(\bm m)+ 1\cdot(v(H)-d_{H}(\bm m))]\\
&=\frac 1 {2^{n-1}}(v(H)-2d_{H}(\bm m)).
\end{aligned}
$$
Plancherel's formula gives us that
$$\frac 1 {2^{n-1}}v(H)=\Vert g\Vert^2=\sum_{\bm m\in\bm 1^\perp}|\hat g(\chi_{\bm m}')|^2\ge \sum_{\bm m\in V(H)}|\hat g(\chi_{\bm m}')|^2.$$
By Cauchy-Schwarz inequality,
\begin{equation*}
		 \begin{split}
		 \sum_{\bm m\in V(H)} \left| \hat{g}(\chi_{\bm m}')\right|^2 &= \sum_{\bm m\in V(H)} \left( \frac{1}{2^{n-1}} \big( v(H)-2d_{H}(\bm m)\big) \right)^2 \\
		                &= \frac{1}{2^{2n-2}} \sum_{\bm m\in V(H)} \big( v(H)-2d_{H}(\bm m)\big)^2\\
		                &\ge \frac{1}{2^{2n-2}} \cdot \frac{1}{v(H)} \left( \sum_{\bm m\in V(H)} \big( v(H)-2d_{H}(\bm m)\big) \right)^2\\
		                &= \frac{1}{2^{2n-2}} \cdot \frac{1}{v(H)} \left( v(H)^2-4e(H) \right)^2.
		 \end{split}
		 \end{equation*}
Combining above two inequalities, we get $$\frac1{2^{n-1}}v(H)\ge \frac1{2^{2n-2}}\frac1{v(H)}(v(H)^2-4e(H))^2,$$ and hence $|v(H)^2-4e(H)|\le 2^\frac {n-1} 2v(H)$.
\end{proof}

$~$

\begin{proof}[Proof of Theorem~\ref{thm_complete_old_eventown}]
Let $n, s$ be positive integers. If $n$ is even, for any even-sized subfamily $\mathcal A\subset 2^{[n]}$ with $|\mathcal A|\ge 2^{\lfloor n\slash2\rfloor}+s=  2^{ n\slash2}+s$, let $H=H(\mathcal A)$ and by Lemma
\ref{lemma_concentration_result} we have
$$4e(H)\ge v(H)^2-2^{n\slash2}v(H)\ge s(2^\frac n 2+s)>s\cdot2^{\lfloor n\slash2\rfloor},$$
and hence $op(\mathcal A)=e(H)>s\cdot2^{\lfloor n\slash2\rfloor-2}$.

If $n$ is odd, for any even-sized subfamily $\mathcal A\subset 2^{[n]}$ with $|\mathcal A|\ge 2^{\lfloor n\slash2\rfloor}+s= 2^{(n-1)\slash2}+s$, let $H=H(\mathcal A)$ and by  Lemma
\ref{lemma_concentration_result_odd},
$$4e(H)\ge v(H)^2-2^{{(n-1)}\slash2}v(H)\ge s(2^\frac {n-1} 2+s)>s\cdot2^{ (n-1)\slash2}.$$
Hence $op(\mathcal A)=e(H)>s\cdot2^{\lfloor n\slash2\rfloor-2}$.
\end{proof}

\section{Conclusion}\label{sec-conc}
We studied the supersaturation problems of oddtown and eventown. It is well known that the maximum size of an oddtown (resp. eventown) family $\mathcal A$ over an $n$ element set is at most $n$ (resp. $2^{\lfloor\frac n 2\rfloor}$). The supersaturation problem counts the number of pairs of subsets with odd-sized intersection in $\mathcal A$ if the size of $\mathcal A$ exceeds the corresponding extremal value.
O'Neill~\cite{o2021towards} initiated the study of this problem and gave two conjectures on the odd pair numbers for oddtown and eventown respectively, and a problem on the asymptotic supersaturation result for eventown. We disproved the conjecture for oddtown, and proved the conjecture for eventown partially when $n$ is large enough.

Asymptotic supersaturation results for the oddtown and eventown subfamilies are given, resulting in different formulas for the minimum odd pair numbers of $\mathcal A$ for different exceeding numbers $s$. {We also completed a result for eventown reaching half of the conjectured lower bound for general $n$ and $s$ proposed by Antipov et al. \cite{antipov2022lovasz}.} Methods like Fourier analysis and extremal graph theory are included. Here we list some open problems.
\begin{itemize}
\item For the supersaturation problem of oddtown family, when $s\le n-4$, the families reaching the tight bounds are not unique under the equivalence of permutation. It is interesting to determine all extremal structures under the equivalence of permutation.
\item In the supersaturation problem of oddtown family, no result about the exact value of the minimum odd pair number is known for $s>n-4$. We believe that the constraint $s\le n-4$ is best possible for the tightness of the bound $op(\mathcal A)\ge s+2$. It is interesting to find the exact values of minimum odd pair numbers for more $s$ systematically.
\item Our eventown supersaturation result Theorem~\ref{thm_for_eventown} only works for sufficiently large $n$, and we do not take efforts to determine the explicit lower bound. It is interesting to give a good explicit lower bound of $n$ such that Theorem~\ref{thm_for_eventown} is satisfied.
\item For Conjecture~\ref{conj_even_town}, we proved that  it is true when $s\leq 2^{\lfloor\frac n 8\rfloor}\slash n$, and left a large gap in the conjectured range $[2^{\lfloor\frac n 2\rfloor}-2^{\lfloor\frac n 4\rfloor}]$ of $s$. We are interested in how to further shrink this gap.
\item For the asymptotic supersaturation results, we have studied the following cases.
    \begin{itemize}
    \item[(1)] Oddtown family, and $\lim_{n\to\infty}s\slash n= c$ for some constant $c$.
    \item[(2)] Oddtown family, with $n=o(s)$ but $\log s= o(n)$.
    \item[(3)] Eventown family, and $s\ge 2^{(1-\epsilon)n}$ for some $\epsilon\in(0, \frac12)$.
    \end{itemize}
    Moreover, remind that Lemma~\ref{lemma_concentration_result} also works for odd-sized family. Consequently, by using the same analysis as in Case (3), one can prove that an odd-sized subfamily shares the same performance on the minimum odd pair number density when the exceeding number $s\ge 2^{(1-\epsilon)n}$ for $\epsilon\in(0, \frac12)$, i.e., $$\lim_{n\to\infty}\min\{\frac{e(H(\mathcal A))}{\binom{|\mathcal A|}2}:\text{ odd-sized }\mathcal A\subset 2^{[n]}; |\mathcal A|\ge 2^{(1-\epsilon)n}\}=\frac12.$$
    Note that for Cases (1) and (2), this minimum density limit is zero for odd-sized family. If $\lim_{n\to\infty}\frac{\log s}n=c$ for some constant $c<\frac12$, by Construction~\ref{cons_A_m_I} with some suitable $m$, we can also determine that the minimum density limit is zero.
    So it is interesting to determine the minimum density limit when $\lim_{n\to\infty}\frac{\log s}n=\frac12$ for odd-sized family.
\item Theorem~\ref{thm_eventownlar} and Theorem~\ref{thm_complete_old_eventown} are proved by using Fourier analysis, which shows the priority of this method. We look forward to more new results in this area derived from Fourier analysis.
\end{itemize}



\vskip 10pt
\bibliographystyle{IEEEtran}
\bibliography{OET}
\end{document}